\documentclass[reqno]{amsart}
\usepackage{amsmath,amsfonts,amssymb,amsthm,epsfig,epstopdf,url,array,comment,hyperref,mathrsfs,mathtools, enumitem}
\usepackage{xparse}
\usepackage{hyperref}
\usepackage{setspace}
\usepackage{graphicx}
\usepackage{float}
\usepackage{tikz-cd}
\usepackage{cancel}
\usepackage[cmtip,all]{xy}
\newcommand{\longsquiggly}{\xymatrix{{}\ar@{<~>}[r]&{}}}
\usepackage{tabularray}

\theoremstyle{plain}
\newtheorem{theorem}{Theorem}[section]
\newtheorem{lemma}[theorem]{Lemma}
\newtheorem{proposition}[theorem]{Proposition}
\newtheorem{corollary}[theorem]{Corollary}

\theoremstyle{definition}
\newtheorem{definition}[theorem]{Definition}
\newtheorem{conjecture}[theorem]{Conjecture}

\newtheorem{examples}[theorem]{Examples}

\newtheorem*{notation}{Notation}
\newtheorem{remark}[theorem]{Remark}

\newcommand{\ind}{\mathbf 1}

\newcommand{\R}{\mathbb{R}}
\renewcommand{\d}{\textnormal{ d}}

\newcommand{\grad}{\nabla}

\newcommand{\ball}{\mathcal{B}}

\newcommand{\hidethis}[1]{}

\newcommand{\E}{\mathbb{E}}
\newcommand{\support}{\textnormal{supp}}

\renewcommand{\div}{\textnormal{div}}

\newcommand{\vol}[1]{\left\lvert #1 \right\rvert}

\newcommand{\energy}[1]{\mathcal{U} \left( #1 \right)}
\RenewDocumentCommand{\energy}{g}{\mathcal{U}\IfValueT{#1}{\left(#1\right)}}
\newcommand{\edensity}[1]{ U \left( #1 \right)}
\RenewDocumentCommand{\edensity}{g}{U\IfValueT{#1}{\left(#1\right)}}
\newcommand{\norm}[1]{\Vert #1 \Vert_{2}}
\newcommand{\pclass}{\mathcal{PC}}

\begin{document}

\title[Gaussian Convolution and the Kneser--Poulsen Conjecture]{Gaussian Convolution, Internal Energies, and the Kneser--Poulsen Conjecture}
\author{Gautam Aishwarya}
\address{Technion -- Israel Institute of Technology, Faculty of Mathematics, Technion City, Haifa 3200003, Israel.}
\email{gautama@technion.ac.il}
\author{Dongbin Li}
\address{Shaanxi Normal University, School of Mathematics and Statistics, Xi'an, 710119, China}
\email{dongbinli@snnu.edu.cn}
\begin{abstract}
We study to what extent the majorisation order between a probability measure and its $1$-Lipschitz image is preserved when both measures undergo Gaussian convolution. We show that the majorisation order is fully preserved in dimensions $n\leq 2$, and obtain dimension-dependent partial preservation in higher dimensions. The perspective taken is that majorisation between densities amounts to comparison through internal energies satisfying a certain pressure condition. Accordingly we introduce a notion of iterated nonnegativity of pressure, closely related to iterated pressures arising in optimal transport, placing the majorisation order within a graded hierarchy of internal energy comparisons. Using the observation that the volume of Euclidean neighbourhoods can be detected from internal energy measurements along the heat flow, we show that our results imply several principal known cases of the Kneser--Poulsen conjecture.

\end{abstract}
\maketitle
\section{Introduction}
\subsection{Internal energy, pressure classes, and the main question}
We think of probability densities $\rho$ on $\R^n$ as different configurations of the same gas. Its \textbf{internal energy} is a functional of the form
\begin{equation}
\energy(\rho)=\int_{\R^n} U(\rho(x))\,\d x,
\end{equation}
where $U:[0,\infty)\to\R$, called the \textbf{energy density}, is a measurable function satisfying $U(0)=0$. For clarity of exposition, throughout the introduction we assume that $\edensity$ is sufficiently differentiable; rigorous definitions without this regularity assumption appear in Section \ref{sec: energypressureprelims}.

The internal energy $\mathcal U$ quantifies the resistance of the configuration to compression through its associated \textbf{thermodynamic pressure} $P$. Borrowing a discussion from \cite[Remarks 5.18]{Villani03} and \cite{McCann97} almost verbatim, consider a gas of mass $m$ distributed uniformly in a region of volume $v$. Its total internal energy is $v\edensity{m/v}$, and hence $\d\mathcal U=-P\d v$ leads to the definition
\begin{equation}\label{eq: defPressureintro} 
P (\rho) = \rho \edensity'(\rho) - \edensity{\rho}. 
\end{equation}

Of course, positive work must be performed to compress a gas by rescaling its volume, thus to be physically realistic, we should assume that internal energy has a \textbf{nonnegative pressure} law. The change-of-variables formula shows that this is exactly the condition ensuring that internal energy increases under sufficiently regular volume contractions; if $\rho_{1} \d x = T_{\#} (\rho_{0} \d x)$ and $T$ is volume contracting, then
\begin{equation}
\energy (\rho_{0}) \leq \energy{\rho_{1}}. 
\end{equation}

A physically realistic model should also require pressure not to decrease as the gas becomes denser. We will therefore be particularly interested in internal energies $\energy$ with a \textbf{nondecreasing pressure} law. Requiring $\rho\mapsto P(\rho)$ to be nondecreasing is equivalent to convexity of the energy density $\edensity$. In particular, such internal energies decrease under free diffusion;
\begin{equation}
\energy{\rho * \gamma_s}\leq \energy{\rho},
\end{equation}
for every density $\rho$ and every $s>0$, where $\gamma_s\propto e^{-\frac{1}{2s}\norm{\cdot}^2}$ denotes the Gaussian density with variance $s$.

We are thus led to two competing effects: volume contraction increases internal energy, while free diffusion decreases it. One might therefore ask whether the energy ordering produced by a volume contraction persists when both configurations are subsequently allowed to diffuse. In dimensions $n\geq 2$, however, such a persistence statement is false for very simple volume contractions. For instance, consider the Boltzmann energy density $\edensity(\rho)=\rho\log\rho$ and deform a standard Gaussian in $\R^2$ by the volume-preserving anisotropic map \begin{equation} T=
    \begin{pmatrix}
        2 & 0\\
        0 & 1/2
    \end{pmatrix}. \end{equation}
The two configurations have the same internal energy before diffusion, whereas after any positive amount of free diffusion, the desired energy ordering is reversed. This failure of persistence is not surprising, since volume contraction records only the distortion of the volume form, whereas free diffusion depends on the full Euclidean metric and hence detects anisotropic distortions invisible to volume contraction. We are led to the following conjecture by strengthening the initial condition.

\begin{conjecture} \label{con: mainheat}
    For every probability measure $\mu$ on $\R^n$, every $1$-Lipschitz map $T:\R^n\to\R^n$, and every $s>0$, we have
    \begin{equation} \label{eq: mainheat}
    \energy{\mu* \gamma_{s}} \leq \energy{T_{\#}\mu * \gamma_{s}},
    \end{equation}
    for every internal energy $\mathcal{U}$ with a \textbf{convex energy density} $\edensity$, that is, with a \textbf{nondecreasing pressure} law. 
\end{conjecture}

Note that Conjecture \ref{con: mainheat} is stated for all probability measures, not just those with densities. This makes perfect sense since $\mu * \gamma_{s}$ is always absolutely continuous no matter how irregular $\mu$ is. In fact, such mollifications of internal energies to admit all probability measures have already been studied, for example, in \cite{CarrilloEspositoWu24}. 

In probability theory and related areas, there is another terminology commonly used for the energy comparisons appearing in Conjecture \ref{con: mainheat}. If two probability densities $\rho$ and $\sigma$ satisfy
\begin{equation}
    \energy{\rho}\leq\energy{\sigma}
\end{equation}
for every internal energy $\energy$ with convex energy density $\edensity$, equivalently with a nondecreasing pressure law, then $\rho$ is said to be \textbf{majorised} by $\sigma$. While majorisation is widely studied (\cite{ShakedShanthikumar07, MarshallOlkinArnold11}), the reader may refer to recent works \cite{MelbourneRoberto23, DePhilippisShenfeld25} for its relationship with transport maps, or \cite[Section 2]{AishwaryaAlamLiMyroshnychenkoZatarainVera23} where it is used to study problems similar to those in the present paper. Thus, Conjecture \ref{con: mainheat} asks whether the majorisation order induced by a $1$-Lipschitz map persists under Gaussian convolution.

The preceding two pressure conditions are the first levels of a natural hierarchy that we introduce in this paper. Since they will depend only on the energy density $\edensity$ and not on the ambient dimension, we regard the pressure classes as classes of energy densities. We have already seen the \textbf{$0$-th pressure class} $\pclass_{0}$, which we define as the class of all energy densities $\edensity$ with pressure 
\begin{equation}
P(\rho) \geq 0. 
\end{equation}
We have also seen the \textbf{$1$-st pressure class} $\pclass_{1}$ whose energy densities are convex, or equivalently for which $P$ is nondecreasing. The fact that $P$ is nondecreasing can be restated as asking 
\begin{equation}
P (c \rho) - P (\rho) \geq 0,
\end{equation}
for all $c \geq 1$. This motivates the definition of the \textbf{$2$-nd pressure class} $\pclass_{2}$ of convex internal energies that we define with the requirement that, 
\begin{equation}
    \left[ P \left( c_{1} c_{2} \rho \right) - P \left( c_{1} \rho \right) \right]
    -
    \left[ P\left( c_{2} \rho \right) - P\left( \rho \right) \right]
    \geq 0,
\end{equation}
for all $c_{1}, c_{2} \geq 1$. 

Higher pressure classes are defined, in Definition \ref{def:pressureclasses}, by iterating this procedure and requiring monotonicity of pressure at the higher orders so obtained. In this introduction, since we assume smoothness of $\edensity$, the condition for an energy density $\edensity$ to belong to \textbf{$k$-th pressure class} $\pclass_{k}$, $k \geq 1$, becomes that $\edensity$ be convex and  
\begin{equation}
    \left(\rho\frac{\d}{\d\rho}\right)^kP(\rho)
    \geq0.
\end{equation}
These classes are closed under positive linear combinations and form a nested hierarchy (see Remark \ref{rem: nesting} and Proposition \ref{prop: regularirregularsame}).

An iteration of pressure closely related to ours appears in optimal transport. Let us briefly comment on it. Writing $D=\rho\frac{\d}{\d\rho}$ for the \emph{Euler operator}, our pressure classes are defined using iterates of $D$ applied to $P$. In optimal transport, it is customary instead to organise the iterated pressures using the shifted operator $D-I$. Thus, $P=P_1=(D-I)\edensity$, while the higher iterated pressures are given by $P_2=(D-I)^2\edensity$, and so on. When the internal energy $\energy$ is differentiated $k$ times along a Wasserstein geodesic, the quantities $P_1,\ldots,P_k$ naturally appear in the resulting expression (for example, see \cite[Chapter 15]{Villani09} for $P_{2}$). By the binomial theorem, these may equivalently be expressed in terms of our iterates $D^0P,\ldots,D^{k-1}P$.

Convexity of internal energy along Wasserstein geodesics is a central theme in optimal transport. Accordingly, the most important pressure condition in this context, discovered by McCann in \cite{McCann97}, involves $P$ and $P_2$ together with the ambient dimension $n$. The resulting classes of internal energies, denoted $\mathcal{DC}_n$ in \cite[Chapter 16]{Villani09}, are dimension dependent. By contrast, the higher-order pressure conditions considered in this paper are dimension independent, although the result we obtain in Theorem \ref{thm: pressureclassmain} depends on the ambient dimension.

\subsection{Persistence of internal energy comparisons under Gaussian convolution}
Our first main result is an affirmative resolution of Conjecture \ref{con: mainheat} in dimension $2$. 
\begin{theorem}\label{thm: maindim2}
    Conjecture \ref{con: mainheat} is true when $n=2$. Thus, for every $1$-Lipschitz map $T: \R^2 \to \R^2$ and every probability measure $\mu \in \mathcal{P}(\R^2$), we have
    \begin{equation}
        \int \edensity (\mu * \gamma_{s}) \d x \leq \int \edensity (T_{\#}\mu * \gamma_{s}) \d x,
    \end{equation}
    whenever $\edensity$ is a convex function with $\edensity(0) = 0$ and $s > 0$.
\end{theorem}

Theorem \ref{thm: maindim2} follows from a more general dimension-dependent version where the pressure hierarchy discussed earlier enters the picture.  

\begin{theorem}\label{thm: pressureclassmain}
    Let $\mu\in\mathcal P(\R^n)$, and suppose $T:\R^n\to\R^n$ is a $1$-Lipschitz map. Then, for all $s>0$,
    \begin{equation} \label{eq: pressureclassmain}
        \energy{\mu*\gamma_s}
        \leq
        \energy{T_{\#}\mu*\gamma_s}
    \end{equation}
    for every internal energy $\energy$ in $\pclass_{k}$, whenever $k \geq \lceil n/2\rceil$.
\end{theorem}
In fact, since $\pclass_{1}$ is exactly the class of nondecreasing pressure, Conjecture \ref{con: mainheat} holds in both dimensions $n=1,2$. We note that the theorem holds in all dimensions for the class $\pclass_{\infty} = \cap_{k=1}^{\infty} \pclass_{k}$. Power laws, which can be seen to be the building blocks of $\pclass_{\infty}$ by the Bernstein's theorem on completely monotone functions, were covered in \cite[Theorem 1.5]{AishwaryaLi25KP} up to a first moment condition on $\mu$. There are other interesting examples which Theorem \ref{thm: pressureclassmain} covers, as discussed in Examples \ref{examples: energies}, such as the Bose--Einstein H-functional in dimensions $n \leq 4$.

For contractions which can themselves be realised through a contracting motion, 
the dimension-dependent pressure condition disappears altogether. Recall that a
map $T:K\to\R^n$ defined on a subset $K \subseteq \R^n$ is a \textbf{continuous contraction} if there exists a homotopy
$\{T_t\}_{t\in[0,1]}$, with $T_0=I$ and $T_1=T$, such that $
    \norm{T_t(x)-T_t(y)}$ is nonincreasing in $t$ for every $x,y\in K$.

\begin{theorem}\label{thm: heatflowpositivepressure}
Let $T:K\to\R^n$ be a continuous contraction, and $\mu\in \mathcal 
{P}(K)$. Then, for any fixed $s>0$, the following hold.
\begin{enumerate}[label=(\roman*),ref=(\roman*)]
\item \begin{enumerate}[label=(\alph*),ref=(\roman{enumi})(\alph*)]
\item\label{item: positivepressuretransport} There exists a coupling
$\pi$ of $\mu * \gamma_s$ and $T_{\#}\mu * \gamma_s$ such that
$(\mu * \gamma_s)(x)
    \leq
    (T_{\#}\mu * \gamma_s)(y)
$
for $\pi$-almost every $(x,y)\in\R^n\times\R^n$.

\item\label{item: positivepressuresmoothtransport} If $T$ can be realised
by continuously contracting maps $\{T_t\}_{t\in[0,1]}$ with $C^1$
trajectories, and $\int \sup_{t\in[0,1]}
    \norm{\frac{\d}{\d t}T_t}^{2}\d\mu<\infty$
then the coupling in \ref{item: positivepressuretransport} can be chosen
to be induced by a volume contracting measurable map $S$ transporting
$\mu*\gamma_s$ to $T_\#\mu*\gamma_s$.
\end{enumerate}
\item\label{item: positivepressuretheorem} We have,
\begin{equation}\label{eq: heatflowpositivepressure}
    \energy{\mu*\gamma_s}
    \leq
    \energy{T_\#\mu*\gamma_s},
\end{equation}
for every internal energy $\energy$ in $\pclass_{0}$, that is, with nonnegative pressure.
\end{enumerate}
\end{theorem}
Note that item \ref{item: positivepressuresmoothtransport}, although not presented in this form, follows from the work in \cite{AishwaryaLi25KP}. Item \ref{item: positivepressuretheorem} follows from this result in the smooth case. For this theorem, our contribution in the present paper is to establish items \ref{item: positivepressuretransport} and \ref{item: positivepressuretheorem}, under the sole assumption of continuity of the trajectories. The proof does not use trajectory fields as in \cite{AishwaryaLi25KP}, and makes the appearance of pressure explicit via a representation formula (equation \eqref{eq: energydifferencepoly}) that might be of further interest.

The appearance of pressure classes in Theorem \ref{thm: pressureclassmain} is explained by a simple interaction between
tensorisation and pressure. Given a probability density $\sigma$ on $\R^m$ and an energy density $\edensity$, we can take the marginalisation of $\edensity$ by $\sigma$,
\begin{equation}
    \Pi_{\sigma}\edensity(\rho)
    =
    \int_{\R^m}\edensity\big(\rho\sigma(y)\big) \d y.
\end{equation}
If an energy $\energy$ has energy density $\edensity$, we will denote by $\Pi_{\sigma}\energy$ the energy with density $\Pi_{\sigma}\edensity$. The key idea involved in proving Theorem \ref{thm: pressureclassmain} from Theorem \ref{thm: heatflowpositivepressure} is that marginalisation by a $2$-dimensional Gaussian introduces additional structure into an internal energy by raising its pressure class by $1$, and all members of the higher class that are non-trivial in our setting arise in this manner. 

\begin{theorem}\label{thm: BCheatflowintro}
    Let $\gamma_s^{(2)}$ denote the two-dimensional Gaussian density of variance $s$. Then, for every $j\geq0$, marginalisation defines a surjective map
    \begin{equation}
        \Pi_{\gamma_s^{(2)}}:
        \left\{
            \widetilde{\edensity}\in\pclass_j:
            \Pi_{\gamma_s^{(2)}}\widetilde{\edensity}
            \text{ is finite valued}
        \right\}
        \twoheadrightarrow
        \left\{
            \edensity\in\pclass_{j+1}:
            \edensity\text{ is continuous at }0
        \right\}.
    \end{equation}
    Moreover, every energy density in the domain is automatically continuous at $0$.
\end{theorem}
The image of $\Pi_{\gamma_s^{(2)}}$ indeed covers all non-trivial internal energies for our purposes because if a convex $\edensity$ is not continuous at $0$, then a simple argument outlined in the proof of Theorem \ref{thm: pressureclassmain} shows that the corresponding internal energy $\energy$ must assign the value $-\infty$ to all bounded strictly positive densities on $\R^n$.

Theorem \ref{thm: pressureclassmain} now follows by combining Theorem
\ref{thm: BCheatflowintro} with a classical lifting trick, sometimes
called the \textit{leapfrog lemma}: every contraction in $\R^n$ can be
realised as a continuous contraction after adding $n$ auxiliary dimensions,
in $\R^{2n}$. Applying Theorem \ref{thm: heatflowpositivepressure} in the
enlarged space and then marginalising the auxiliary Gaussian coordinates
in pairs, with one additional dummy dimension when $n$ is odd, raises the
pressure class by $\lceil n/2\rceil$. The same reasoning shows more generally that if a particular $1$-Lipschitz
map $T$ on $\R^n$ can be lifted to a continuous contraction by adding at
most two auxiliary dimensions, then Conjecture \ref{con: mainheat} holds
for $T$.  

\subsection{Applications to the Kneser--Poulsen problem}
The questions explored in the present work were first motivated by the following problem in discrete geometry, attributed to Kneser \cite{Kneser55} and Poulsen \cite{Poulson54}.

\begin{conjecture}[Kneser--Poulsen conjecture]\label{con: KPsameradii}
Let $\{x_{1} , \ldots , x_{k}\}$ and $\{y_{1}, \cdots , y_{k}\}$ be two sets of points in $\mathbb{R}^n$ such that $\Vert y_{i} - y_{j} \Vert_{2} \leq \Vert x_{i} - x_{j} \Vert_{2}$ for all $i, j \in \{1, \ldots, k\}$. If $r> 0$, then we have
\begin{equation} \label{eq: KPsameradii}
\vol{\bigcup_{i=1}^{k} \ball (y_{i}, r)} \leq \vol{ \bigcup_{i=1}^{k} \ball (x_{i}, r)},
\end{equation}
where $\vol{\cdot}$ is the Lebesgue measure on $(\R^n, \Vert \cdot \Vert_{2})$ and $\mathcal{B}(x,r)$ is the closed ball of radius $r$ centred at $x$.
\end{conjecture}
The Kneser--Poulsen conjecture is known in full generality only in $\R^2$ (\cite{BezdekConnelly02}). The proof of Bezdek and Connelly in \cite{BezdekConnelly02} crucially uses a result of Csik\'os (\cite{Csiskos98}) establishing Conjecture \ref{con: KPsameradii} for continuous contractions.  Beyond the plane, it is known in special cases (see, for example, \cite{BezdekNaszodi18,Gorbovickis18}). The reader may consult excellent survey articles \cite{KuperbergToth22,BezdekLangiNaszodi26} for a more complete bibliography of contributions, particularly those preceding Csik\'os's result (such as \cite{Bollobas68,BernSahai98,Capoyleas96}).

By approximating compact sets from within by finite sets, and observing that
$x_i \mapsto y_i$ defines a $1$-Lipschitz map, one can equivalently restate
Conjecture \ref{con: KPsameradii} in the following form.

\begin{conjecture} \label{con: KPtubeform}
For every contraction (that is, a $1$-Lipschitz map)
$T:(K,\Vert\cdot\Vert_2)\to(\R^n,\Vert\cdot\Vert_2)$ defined on a compact set
$K\subseteq\R^n$, and every $r>0$, we have
\begin{equation} \label{eq: KPtubeform}
    \vol{T[K]+r\ball}\leq \vol{K+r\ball},
\end{equation}
where $\ball=\ball(0,1)$ denotes the Euclidean unit ball.
\end{conjecture}

This formulation also suggests variants of the Kneser--Poulsen problem in which one imposes geometric structure on the set $K$. To the best of our knowledge, relatively little is known in this direction, even for natural classes such as convex bodies or compact submanifolds of $\R^n$, when the map $T$ is allowed to be an arbitrary contraction.

The internal energy inequalities we develop have direct geometric consequences. The two pressure conditions lead to two different conclusions: nondecreasing pressure recovers the Kneser--Poulsen inequality, while the larger class of nonnegative pressures yields volume of annular regions.

\begin{theorem} \label{thm: heattoKPintro}
Let $K\subseteq\R^n$ be compact and let $T:K\to\R^n$ be a $1$-Lipschitz map.
\begin{enumerate}[label=(\roman*),ref=(\roman*)]
    \item Suppose that
    \begin{equation}
        \energy{\mu*\gamma_s}
        \leq
        \energy{T_\#\mu*\gamma_s}
    \end{equation}
    for every $\mu \in \mathcal{P}(K)$, every $s>0$, and every internal energy
    $\energy$ in $\pclass_{1}$, that is, with \textbf{nondecreasing pressure}. Then
    \begin{equation}
        \vol{T[K]+r\ball}
        \leq
        \vol{K+r\ball}
    \end{equation}
    for every $r>0$.
    \item If the same internal energy inequality holds for every internal energy $\energy$ in $\pclass_{0}$, that is, with \textbf{nonnegative pressure}, then
    \begin{equation}
        \vol{
        \left(T[K]+r_2\ball\right)
        \setminus
        \left(T[K]+r_1\ball\right)}
        \leq
        \vol{
        \left(K+r_2\ball\right)
        \setminus
        \left(K+r_1\ball\right)}
    \end{equation}
    for every $0<r_{1}<r_{2}$.
\end{enumerate}
\end{theorem}
\begin{remark} \label{rem: KPforvaryingradiiIntro}
    We prove a more general version of the above theorem in which the constant radius is replaced by an arbitrary continuous function $r:K\to(0,\infty)$. Moreover, in the constant radius case, it suffices to assume the energy inequalities for a single probability measure
    $\mu\in\mathcal{P}(K)$ with full support.
\end{remark}

The proof of this result is based on the observation that the volume of every neighbourhood of the support $K$ of $\mu$ can be recovered from the knowledge of $\int \edensity{\mu * \gamma_{s}} \d x$ for all convex $\edensity$ and all small $s>0$. More precisely, in the setting of Theorem \ref{thm: heattoKPintro}, suppose $\gamma_{s} = C_{s} e^{- \frac{1}{2s}\norm{x}^2}$, and $\edensity_{s}(\rho) = - \min \{ C^{-1}_{s} e^{\frac{1}{2s}r^2} \rho , 1 \}$. Then we show that $ - \edensity_{s} (\mu * \gamma_{s})$ appropriately converges to the indicator of $K + r \ball$.

The manner in which we exploit the encoding of the distance function in the heat kernel is more general; see, for example, the pioneering work \cite{Varadhan67}. For this reason we believe that our methods could be helpful in exploring Kneser--Poulsen-type inequalities beyond Euclidean spaces, such as on Riemannian manifolds, although there are curvature obstructions to Conjecture \ref{con: KPtubeform} holding in full generality in this setting (\cite{CsikosKunszenti-Kovacs10, CsikosHorvath14}). We also note that the use of Gaussian convolution to recover geometric quantities in Euclidean space has a rich lineage; for instance, see \cite{Ledoux94}. 

Using the full version of Theorem \ref{thm: heattoKPintro} we can get, for example, the following geometric results in a straightforward way. 

\begin{corollary}\label{cor: classicalKPconsequencesIntro}
\begin{enumerate}[label=(\roman*),ref=(\roman*)]
\item[] 
    \item\label{item: CsikosKPIntro}
    Let $x_1,\ldots,x_k\in\R^n$ undergo a continuous contraction to
    $y_1,\ldots,y_k\in\R^n$. Then, for arbitrary $r_1,\ldots,r_k>0$,
    \begin{equation}
        \vol{\bigcup_{i=1}^{k}\ball(y_i,r_i)}
        \leq
        \vol{\bigcup_{i=1}^{k}\ball(x_i,r_i)}.
    \end{equation}
    This recovers the Kneser--Poulsen theorem for continuous contractions
    with arbitrary radii, originally proved by Csik\'os \cite{Csiskos98}.

    \item\label{item: surfaceKPIntro}
Suppose $n\geq 2$. Let $x_1,\ldots,x_k\in\R^n$ undergo a continuous contraction to
$y_1,\ldots,y_k\in\R^n$, and let $r_1,\ldots,r_k>0$. For each $i$, let
$F_i^x\subseteq\partial\ball(x_i,r_i)$ and
$F_i^y\subseteq\partial\ball(y_i,r_i)$ denote the exposed parts of the respective spheres, namely the portions lying on the boundary of the corresponding union of balls. If several of the spheres coincide, we assign their common exposed part to a single index, so that it is counted only once. Then
\begin{equation}
    \sum_{i=1}^{k}
    \frac{\mathcal H^{n-1}(F_i^y)}{r_i}
    \leq
    \sum_{i=1}^{k}
    \frac{\mathcal H^{n-1}(F_i^x)}{r_i}.
\end{equation}
In particular, when $r_1=\cdots=r_k=r$, the exposed parts partition the boundary up to an $\mathcal H^{n-1}$-null set, and hence the ordinary surface area of the union is nonincreasing under the continuous contraction. This recovers the corresponding result of Bezdek and Connelly \cite[Theorem~3]{BezdekConnelly02}.

    \item\label{item: planarKPIntro}
    Let $x_1,\ldots,x_k,y_1,\ldots,y_k\in\R^2$ satisfy
    \begin{equation}
        \norm{y_i-y_j}
        \leq
        \norm{x_i-x_j}
    \end{equation}
    for every $i,j$. Then, for arbitrary $r_1,\ldots,r_k>0$,
    \begin{equation}
        \vol{\bigcup_{i=1}^{k}\ball(y_i,r_i)}
        \leq
        \vol{\bigcup_{i=1}^{k}\ball(x_i,r_i)}.
    \end{equation}
    This recovers the union part of the planar Kneser--Poulsen theorem
    of Bezdek and Connelly \cite[Corollary~1]{BezdekConnelly02}.
\end{enumerate}
\end{corollary}
The same framework also recovers further known results, such as the Kneser--Poulsen theorem for coordinatewise contractions (\cite[Theorem~1.3]{BezdekNaszodi18}). We do not pursue these extensions here, since the examples above already illustrate the geometric consequences of the internal energy inequalities that are most relevant to the present work.

\begin{remark}\label{rem: whynotnonnegativePressureinanyT}
Conjecture \ref{con: mainheat} cannot be strengthened by replacing nondecreasing pressure with nonnegative pressure. Indeed, such a strengthening, together with Theorem \ref{thm: heattoKPintro}, would imply monotonicity of the surface area of unions of congruent balls under arbitrary $1$-Lipschitz maps. This fails already in the plane as exhibited by the example of Habicht and Kneser reproduced in \cite[Figure 2]{BezdekConnelly02}. Thus, the distinction between the two pressure classes is reflected in the corresponding geometry.
\end{remark}

\subsection{Gaussian characterisation}

We now turn to a natural question underlying the preceding discussion. Why the Gaussian? It will not escape the curious reader's notice that neither Theorem \ref{thm: BCheatflowintro} nor Theorem \ref{thm: heattoKPintro} by itself seems to single out Gaussian convolution, and analogous statements may hold for slightly more general kernels. Gaussian convolution is exceptionally well behaved with respect to a wide range of analytic tools, including Fourier analysis, semigroup methods, and stochastic analysis. This flexibility is one reason to hope that the difficulty posed by arbitrary contractions, perhaps the main obstruction to the Kneser--Poulsen conjecture, may be more tractable in the present framework.

Below we present more concrete evidence for the special role of the Gaussian kernel in this problem. In the smooth continuous contraction argument of \cite{AishwaryaLi25KP}, if $\mu_t=(T_t)_\#\mu$ has velocity field $v_t$, then
\begin{equation}
    \widetilde v_t
    =
    \frac{(v_t\mu_t)*g}{\mu_t*g}
\end{equation}
is a canonical velocity field for the convolved curve $\{\mu_t*g\}_{t\in[0,1]}$. When $g$ is Gaussian, it is shown that $\div\widetilde v_t\leq0$, which yields the volume-contracting transport in item \ref{item: positivepressuresmoothtransport}. The following result shows that, in dimensions $n\geq2$, this infinitesimal volume contraction characterises the Gaussian.

\begin{theorem}\label{thm: GaussianCharacterisationIntro}
    Let $n\geq2$, and let $g$ be a positive $C^2$ probability density on $\R^n$. Then the following are equivalent.
    \begin{enumerate}[label=(\roman*),ref=(\roman*)]
        \item\label{item: ourdispersivecondition} For every compactly supported probability measure $\mu$ and every smooth continuous contraction $\{T_t\}_{t\in[0,1]}$, the canonical velocity field $\widetilde v_t$ satisfies
        \begin{equation}
            \div\widetilde v_t\leq0.
        \end{equation}
        \item\label{item: itisgaussian} Up to translation, $g$ is an isotropic Gaussian density.
    \end{enumerate}
\end{theorem}

In fact, the implication \ref{item: ourdispersivecondition} $\Rightarrow$ \ref{item: itisgaussian} already follows by testing \ref{item: ourdispersivecondition} on probability measures supported on two points. We expect that a stronger rigidity statement holds at the level of internal energies themselves; namely, that in dimensions $n\geq2$, the validity of the nonnegative-pressure comparison in Theorem \ref{thm: heatflowpositivepressure} for every continuous contraction characterises the isotropic Gaussian among convolution kernels. Theorem \ref{thm: GaussianCharacterisationIntro} establishes this under the stronger pointwise condition on the canonical velocity field. 

There is a useful parallel with the theory of dispersive order in statistics (see \cite[Section~3.B]{ShakedShanthikumar07}). In one dimension, the condition $\div\widetilde v_t\leq0$ simply means that $\widetilde v_t$ is nonincreasing, and hence that its flow is contractive. Moreover, the same argument works for every log-concave convolution kernel. This is consistent with the theorem of Lewis and Thompson \cite[Theorem~8]{LewisThompson81}, which characterises the probability densities whose convolution preserves dispersive order as precisely those with log-concave densities. Theorem \ref{thm: GaussianCharacterisationIntro} shows that the corresponding infinitesimal volume-contraction property is considerably more rigid in dimensions $n\geq2$, as it singles out the isotropic Gaussian.
\subsection{Organisation of the paper} 
We define the pressure classes without appealing to differentiability, and give some examples, in Section \ref{sec: energypressureprelims}. In Section \ref{sec: emonotonicity}, we prove Theorems \ref{thm: heatflowpositivepressure} and \ref{thm: GaussianCharacterisationIntro}. The interaction between Gaussian marginalisation and pressure classes, culminating in Theorems \ref{thm: BCheatflowintro} and \ref{thm: pressureclassmain}, is developed in Section \ref{sec: BezdekConnelly}. Finally, in Section \ref{sec: KP}, we prove the full version of Theorem \ref{thm: heattoKPintro} concluding the paper.

\begin{notation}
We will identify a probability measure with its density whenever there is no scope for confusion. For example, $\gamma_s$ will denote both the isotropic Gaussian measure and its density. Probability densities will typically be denoted by $\rho,\sigma$, etc., while probability measures not necessarily absolutely continuous with respect to Lebesgue measure will be denoted by $\mu,\nu$, etc. We write $\mathcal P(E)$ for the space of probability measures on a set $E$, and $\mathcal{P}_{ac}(E)$ for those which are absolutely continuous with respect to Lebesgue measure.
\end{notation}

\textbf{Acknowledgements.} The authors are grateful to Shiri Artstein-Avidan for suggesting the possibility that Gaussian convolution alone might be sufficient to approach the Kneser--Poulsen conjecture, a suggestion that inspired the present work. The first-named author would like to thank Jun Kitagawa and Sayantan Chakraborty for many enlightening discussions on a related project, which helped lead the authors to the Gaussian characterisation.

\textbf{Statement on AI use.} The main conceptual framework and principal results of this paper were developed by the authors. ChatGPT 5.6 was subsequently used to assist in developing proofs and technical details based on ideas provided by the authors. In addition, interactions with ChatGPT 5.6 led to the variable-radii extension described in Remark \ref{rem: KPforvaryingradiiIntro} and were instrumental in developing the proof of item \ref{item: positivepressuretransport} of Theorem \ref{thm: heatflowpositivepressure}. The manuscript itself was written and organised by the authors; ChatGPT was used as an auxiliary mathematical and editorial tool. All AI-assisted material included in the paper was reviewed and verified by the authors, who take full responsibility for its content.

\section{Internal energies and pressure classes} \label{sec: energypressureprelims}
We start by formally recording the definition of internal energy.
\begin{definition}[Internal energy and energy density]
    Let $\edensity:[0,\infty)\to\R$ be a measurable function satisfying $\edensity(0)=0$. For every $\rho\in\mathcal P_{ac}(\R^n)$ for which the integral below is defined, let
\begin{equation}
    \energy(\rho)
    =
    \int_{\R^n}\edensity(\rho(x))\,\d x.
\end{equation}
The value of $\energy(\rho)$ may be finite, $+\infty$, or $-\infty$. We call $\edensity$ the \textbf{energy density} of $\energy$.
\end{definition}
If $\edensity$ were differentiable as assumed in the introduction, and we defined its pressure $P$ by \eqref{eq: defPressureintro} then, for every $\rho>0$,
\begin{equation} \label{eq: positivepressureinU}
    \frac{\d}{\d\rho}
    \left(\frac{\edensity(\rho)}{\rho}\right)
    =
    \frac{\rho\edensity'(\rho)-\edensity(\rho)}{\rho^2}
    =
    \frac{P(\rho)}{\rho^2}.
\end{equation}
Consequently, $P\geq0$ on $[0,\infty)$ if and only if $\rho\mapsto\edensity(\rho)/\rho$ is nondecreasing on $(0,\infty)$. This motivates the following definition, which does not require differentiability.

\begin{definition}[Nonnegative pressure]
    An energy density $\edensity$, and the corresponding internal energy $\energy$, are said to have a \textbf{nonnegative pressure} law if
    \begin{equation}
        \rho\longmapsto\frac{\edensity(\rho)}{\rho}
    \end{equation}
    is nondecreasing on $(0,\infty)$.
\end{definition}
Continuing with the ideal case discussed in the introduction, note further that $P$ would be nondecreasing if and only if
\begin{equation}
    P'(\rho)
    =
    \rho\edensity''(\rho)
    \geq0,
    \qquad \rho>0,
\end{equation}
that is, if and only if $\edensity$ were convex. This allows us to define the nondecreasing-pressure condition without appealing to the existence of a pressure function.

\begin{definition}[Nondecreasing pressure]
    An energy density $\edensity$, and the corresponding internal energy $\energy$, are said to have a \textbf{nondecreasing pressure} law if $\edensity$ is convex on $[0,\infty)$.
\end{definition}
Once $\edensity$ is convex, its one-sided derivatives exist at every $\rho>0$. We use the right derivative to define its pressure law.

\begin{definition}[Pressure]
    Suppose that the energy density $\edensity$ is convex. Its \textbf{pressure law} is the function $P:[0,\infty)\to\R$ defined by $P(0)=0$ and
    \begin{equation}\label{eq: defPressure}
        P(\rho)
        =
        \rho\edensity'_+(\rho)-\edensity(\rho),
        \qquad \rho>0,
    \end{equation}
    where $\edensity'_+$ denotes the right derivative of $\edensity$.
\end{definition}

With a pressure function in hand, we are ready to define the higher pressure classes. For $c\geq1$ and a function $F:[0,\infty)\to\R$, define its multiplicative forward difference by
\begin{equation}
    \Delta_cF(\rho)
    =
    F(c\rho)-F(\rho).
\end{equation}

\begin{definition}[Pressure classes] \label{def:pressureclasses}
    The class $\pclass_0$ consists of the energy densities with nonnegative pressure. For an integer $k\geq1$, a convex energy density $\edensity$ belongs to the \textbf{$k$-th pressure class} $\pclass_k$ if its pressure law satisfies
    \begin{equation}
        \Delta_{c_1}\cdots\Delta_{c_k}P(\rho)
        \geq0
    \end{equation}
    for every $\rho\geq0$ and every $c_1,\ldots,c_k\geq1$. An internal energy belongs to $\pclass_k$ if its energy density does.
\end{definition}
\begin{remark} \label{rem: nesting}
The highest-order condition in the definition implies all the lower-order finite-difference conditions. Indeed, suppose $k\geq2$, and the energy density $\edensity$ with associated pressure $P$ is in $\pclass_{k}$. Fix $c_1,\ldots,c_{k-1}\geq1$. Then,
\begin{equation}
    F(\rho)
    =
    \Delta_{c_1}\cdots\Delta_{c_{k-1}}P(\rho)
\end{equation}
satisfies $\Delta_cF(\rho)\geq0$ for every $c\geq1$. This means that $F(\rho)$ is nondecreasing. Now, because $P$ is nonnegative and nondecreasing, the limit $\lim_{\rho \to 0} P(\rho )$ exists. Given that $F$ is formed of iterations of differences of $P$, this implies $\lim_{\rho \to 0}F(\rho) = 0$, thereby proving $F \geq 0$ exactly as required by the $\pclass_{k-1}$ condition.  
\end{remark}

The definition above agrees with the one in the introduction once sufficient derivatives of $\edensity$ are assumed.

\begin{proposition}\label{prop: regularirregularsame}
    Let $k\geq1$, and suppose that the energy density $\edensity$ is convex on $[0, \infty)$ and $C^{k+1}(0,\infty)$. Define its pressure $P$ as in \eqref{eq: defPressureintro}. Then $\edensity\in\pclass_k$ if and only if
    \begin{equation}
        \left(\rho\frac{\d}{\d\rho}\right)^kP(\rho)
        \geq0
    \end{equation}
    for every $\rho>0$.
\end{proposition}

\begin{proof}
For any differentiable function $F$,
\begin{equation}
    \Delta_cF(\rho)\geq0
    \quad\text{for every }c\geq1
\end{equation}
if and only if $F$ is nondecreasing, or equivalently,
\begin{equation}
    \rho\frac{\d}{\d\rho}F(\rho)\geq0.
\end{equation}
Moreover,
\begin{equation}
    \rho\frac{\d}{\d\rho}\Delta_cF(\rho)
    =
    c\rho F'(c\rho)-\rho F'(\rho)
    =
    \Delta_c\left(\rho\frac{\d}{\d\rho}F(\rho)\right).
\end{equation}

Iterating this observation gives
\begin{equation}
    \Delta_{c_1}\cdots\Delta_{c_k}P(\rho)\geq0
    \quad\text{for every }c_1,\ldots,c_k\geq1,
\end{equation}
if and only if,
\begin{equation}
    \left(\rho\frac{\d}{\d\rho}\right)^kP(\rho)\geq0.
\end{equation}
\end{proof}

We end this section by discussing some examples. 
\begin{examples} \label{examples: energies}
\begin{enumerate}[label=(\roman*),ref=(\roman*)]
\item[]
\item \textbf{Boltzmann $H$-functional.}
For
\begin{equation}
    \edensity(\rho)=\rho\log\rho,
\end{equation}
with $\edensity(0)=0$, the pressure is $P(\rho)=\rho$. Hence
    $\left(\rho\frac{\d}{\d\rho}\right)^kP(\rho)=\rho$
for every $k\geq1$. Thus, the Boltzmann energy belongs to every pressure class $\pclass_k$. Consequently, Theorem \ref{thm: pressureclassmain} applies in every dimension in this case.

\item \textbf{R\'enyi entropies.}
For $\alpha>0$, $\alpha\neq1$, consider the energy density
\begin{equation}
    \edensity_\alpha(\rho)
    =
    \frac{\rho^\alpha-\rho}{\alpha-1}.
\end{equation}
Since $\rho$ integrates to one, the corresponding internal energy is the negative Tsallis entropy of order $\alpha$, and is a monotone function of the negative R\'enyi entropy of the same order. Its pressure is $P_\alpha(\rho)=\rho^\alpha$,
and therefore $\left(\rho\frac{\d}{\d\rho}\right)^kP_\alpha(\rho)
    =
    \alpha^k\rho^\alpha>0$. Thus, these energies also belong to every $\pclass_k$. Consequently, Theorem \ref{thm: pressureclassmain} applies in every dimension in this case as well.

\item \textbf{Hinge functions.}
For $a>0$, 
\begin{equation}
    \edensity_a(\rho)
    =
    \frac{(\rho-a)_+}{a}.
\end{equation}
 is convex and hence in $\pclass_1$. The associated pressure is $\ind_{\{\rho\geq a\}}$. In fact, comparisons for all energies in $\pclass_1$ can be reduced to hinge functions; see \cite[Theorem 2.1]{MelbourneRoberto23}. Consequently, Theorem \ref{thm: pressureclassmain} applies in dimensions $n\leq2$ in this case.

\item \textbf{Differences of hinges.}
Let $0<a<b$ and set
\begin{equation}
    \edensity_{a,b}
    =
    \edensity_a-\edensity_b.
\end{equation}
Then $\edensity_{a,b}(\rho)/\rho$ is nondecreasing, and hence $\edensity_{a,b}$ is in $\pclass_0$. However, $\edensity_{a,b}$ is not convex, so it is not in $\pclass_{1}$. Thus, Theorem \ref{thm: pressureclassmain} does not apply to this example in any dimension, and for good reason, as explained in Remark \ref{rem: whynotnonnegativePressureinanyT}.
\item \textbf{Bose--Einstein H-functional.}
Consider the energy density
\begin{equation}
    \edensity(\rho)
    =
    \rho\log\rho-(1+\rho)\log(1+\rho),
\end{equation}
with $\edensity(0)=0$. This is the negative Bose--Einstein entropy density. Its pressure is $P(\rho)=\log(1+\rho)$. Moreover,
$\left(\rho\frac{\d}{\d\rho}\right)P(\rho)=\frac{\rho}{1+\rho}$ and
$\left(\rho\frac{\d}{\d\rho}\right)^2P(\rho)=\frac{\rho}{(1+\rho)^2}$,
whereas
$\left(\rho\frac{\d}{\d\rho}\right)^3P(\rho)=\frac{\rho(1-\rho)}{(1+\rho)^3}$.
Thus the Bose--Einstein energy belongs to $\pclass_2$, but not to $\pclass_3$. Consequently, Theorem \ref{thm: pressureclassmain} applies in dimensions $n\leq4$ in this case.
\end{enumerate}
\end{examples}
\section{Energy monotonicity for continuous contractions under Gaussian convolution} \label{sec: emonotonicity}

We first reduce the proof to polynomial energy densities. For this reduction, it would be convenient to first interpret the order induced by internal energies in $\pclass_{0}$ probabilistically. 
\begin{definition}[Stochastic order]
A random variable $X$ is said to be \textbf{stochastically dominated} by a random variable $Y$ if $\E Q (X) \leq \E Q(Y)$ for every nondecreasing function $Q: \R \to \R$. In this case, we write $X \preceq_{st} Y$. 
\end{definition}
The following lemma clarifies how $\pclass_{0}$ internal energy comparisons relate to the stochastic ordering.
\begin{lemma} \label{lem: pc0stochastictransport}
Let $\rho$ and $\sigma$ be probability densities on $\R^n$. Let $X$ and $Y$ denote $\R^n$-valued random vectors having densities $\rho$ and $\sigma$, respectively. Then the following statements are equivalent. 
\begin{enumerate}[label=(\roman*),ref=(\roman*)]
    \item $\energy{\rho}\leq\energy{\sigma}$ for every internal energy in $\pclass_{0}$ for which the two sides are defined.
    \item We have,
    \begin{equation}
        \rho(X) \preceq_{st} \sigma(Y).
    \end{equation}
    \item There exists a joint distribution $( X, Y)$ such that
    \begin{equation}
        \rho(X) \leq \sigma(Y)
    \end{equation}
    almost surely.
\end{enumerate}
\end{lemma}
\begin{proof}
For $a>0$, write $\edensity(a)=aQ(a)$. Then $\edensity\in\pclass_0$ if and only if $Q$ is nondecreasing. We also have,
\begin{equation}
    \energy{\rho}
    =
    \E Q(\rho(X)),
    \qquad
    \energy{\sigma}
    =
    \E Q(\sigma(Y)).
\end{equation}
Thus, by the standard characterisation of the stochastic order
\cite[Section~1.A.1]{ShakedShanthikumar07}, \textup{(i)} and \textup{(ii)}
are equivalent. Condition \textup{(iii)} also immediately implies \textup{(i)} similarly.

By \cite[Theorem~1.A.1]{ShakedShanthikumar07}, \textup{(ii)} is equivalent
to the existence of a coupling of the random variables $\rho(X)$ and
$\sigma(Y)$ for which
\begin{equation}
    \rho(X)\leq\sigma(Y)
\end{equation}
almost surely. The function $\rho$ induces a disintegration of $X$ into conditionals $X_{a} \sim X \vert \rho(X) = a$ and likewise $\sigma$ induces a disintegration of $Y$ into conditionals $Y_{b} \sim Y \vert \sigma (Y) = b$. Conditioned on $(\rho(X), \sigma(Y)) = (a,b)$, couple $X_{a}$ and $Y_{b}$ independently. Averaging these conditionals over $(\rho(X) , \sigma(Y))$ results in a coupling of $X$ and $Y$ satisfying $\rho(X)\leq\sigma(Y)$ almost surely. This proves that \textup{(ii)} implies \textup{(iii)}.
\end{proof}

\begin{lemma} \label{lem: polyareenough}
    Let $\rho,\sigma \in\mathcal P(\R^n)$ be bounded densities satisfying $0\leq\rho,\sigma \leq C$. To prove $\energy{\rho}\leq\energy{\sigma}$ for every internal energy $\energy$ in $\pclass_{0}$, it suffices to prove the inequality for every $\energy$ with polynomial energy density $\edensity$ whose associated pressure satisfies $P\geq0$ on $[0,C]$.
\end{lemma}

\begin{proof}
Let $X$ and $Y$ have densities $\rho$ and $\sigma$, respectively. By the preceding lemma, it suffices to prove
\begin{equation}
    \rho(X)\preceq_{st}\sigma(Y).
\end{equation}
It is enough to test this order against continuous nondecreasing functions $Q$ on $[0,C]$, since these already characterise the stochastic order (see, for example, the discussion in \cite[Section 1.A.1]{ShakedShanthikumar07}). Let $\{Q_j \}$ denote the approximation of such a function $Q$ by Bernstein polynomials. Then each $Q_j$ is nondecreasing on $[0,C]$ and $Q_j\to Q$, in fact, uniformly. Moreover,
\begin{equation}
    \edensity_j(a)=aQ_j(a)
\end{equation}
is a polynomial energy density whose pressure is
\begin{equation}
    P_j(a)=a^2Q_j'(a)\geq0
\end{equation}
on $[0,C]$. Hence, by assumption,
\begin{equation}
    \E Q_j(\rho(X))
    \leq
    \E Q_j(\sigma(Y)).
\end{equation}
Letting $j\to\infty$ gives the same inequality with $Q$ in place of $Q_j$, and therefore
$\rho(X)\preceq_{st}\sigma(Y)$. The preceding lemma completes the proof.
\end{proof}
We continue to write our proofs in probabilistic language for the next theorem. 
\begin{proof}[Proof of Theorem \ref{thm: heatflowpositivepressure}]
Let $\{T_t\}_{t\in[0,1]}$ realise the continuous contraction, with $T_0=I$ and $T_1=T$, and let $X$ be a random vector with distribution $\mu$. Set
\begin{equation}
    \mu_t=(T_t)_\#\mu.
\end{equation}
Recall that
\begin{equation}
    \gamma_s(x)
    =
    C e^{-\frac{1}{2s}\norm{x}^2},
    \qquad
    C=\gamma_s(0)=(2\pi s)^{-n/2}.
\end{equation}
Then
\begin{equation}
    (\mu_t*\gamma_s)(x)
    =
    \E\gamma_s(x-T_t(X)),
\end{equation}
and hence $0<\mu_t*\gamma_s\leq C$. Thus Lemma \ref{lem: polyareenough} is applicable.

We begin with the monomial energy density $\edensity(\rho)=\rho^m$, where $m\geq2$ is an integer. Let $X_1,\ldots,X_m$ be independent copies of $X$. Then
\begin{equation}
\begin{split}
    (\mu_t*\gamma_s)^{m}(x)
    &=
    \E\prod_{i=1}^m\gamma_s(x-T_t(X_i))\\
    &=
    C^m
    \E
    e^{-\frac{1}{2s}\sum_{i=1}^m\norm{x-T_t(X_i)}^2}.
\end{split}
\end{equation}
For $y_1,\ldots,y_m\in\R^n$, with $\bar y=\frac{1}{m}\sum_i y_i$, we have
\begin{equation}
    \sum_{i=1}^m\norm{x-y_i}^2
    =
    m\norm{x-\bar y}^2
    +
    \frac1m\sum_{i<j}\norm{y_i-y_j}^2.
\end{equation}
Since
\begin{equation}
    \int_{\R^n}
    e^{-\frac{m}{2s}\norm{x-\bar y}^2}\d x
    =
    \frac{1}{Cm^{n/2}},
\end{equation}
we obtain
\begin{equation}\label{eq: mRenyirep}
    \int_{\R^n}(\mu_t*\gamma_s)^{m}(x)\d x
    =
    \frac{C^{m-1}}{m^{n/2}}
    \E
    e^{-\frac{1}{2ms}\sum_{i<j}
    \norm{T_t(X_i)-T_t(X_j)}^2}.
\end{equation}

For each $i<j$, let $\nu_{ij}$ be the random Lebesgue--Stieltjes measure induced by the continuous nondecreasing function
\begin{equation}
    t\longmapsto
    -\norm{T_t(X_i)-T_t(X_j)}^2.
\end{equation}
The Lebesgue--Stieltjes chain rule gives
\begin{equation}
\begin{split}
&
e^{-\frac{1}{2ms}\sum_{i<j}\norm{T(X_i)-T(X_j)}^2}
-
e^{-\frac{1}{2ms}\sum_{i<j}\norm{X_i-X_j}^2}\\
&=
\frac{1}{2ms}
\sum_{i<j}
\int_{(0,1]}
e^{-\frac{1}{2ms}\sum_{k<l}\norm{T_t(X_k)-T_t(X_l)}^2}
\d\nu_{ij}(t).
\end{split}
\end{equation}
Substituting this into \eqref{eq: mRenyirep} and using the exchangeability of $X_1,\ldots,X_m$, together with
$\frac{1}{2m}\binom{m}{2}=\frac{1}{4}(m-1)$, we obtain
\begin{equation}
\begin{split}
&
\int_{\R^n}(\mu_1*\gamma_s)^{m}(x)\d x
-
\int_{\R^n}(\mu_0*\gamma_s)^{m}(x)\d x\\
&=
\frac{m-1}{4s}
\E\int_{(0,1]}
\frac{C^{m-1}}{m^{n/2}}
e^{-\frac{1}{2ms}\sum_{k<l}\norm{T_t(X_k)-T_t(X_l)}^2}
\d\nu_{12}(t).
\end{split}
\end{equation}
Reassembling the Gaussian product as in \eqref{eq: mRenyirep}, and then conditioning on $X_1,X_2$, gives
\begin{equation}
\begin{split}
&
\int_{\R^n}(\mu_1*\gamma_s)^{m}(x)\d x
-
\int_{\R^n}(\mu_0*\gamma_s)^{m}(x)\d x\\
&=
\frac{m-1}{4s}
\E\int_{(0,1]}
\int_{\R^n}
(\mu_t*\gamma_s)^{m-2}(x)
\prod_{i=1}^2
\gamma_s(x-T_t(X_i))
\d x\,\d\nu_{12}(t).
\end{split}
\end{equation}
The pressure corresponding to $\edensity(\rho)=\rho^m$ is $P(\rho)=(m-1)\rho^m$. Hence
\begin{equation}\label{eq: energydifferencepoly}
\begin{split}
&
\energy{\mu_1*\gamma_s}
-
\energy{\mu_0*\gamma_s}\\
&=
\frac{1}{4s}
\E\int_{(0,1]}
\int_{\R^n}
\frac{P((\mu_t*\gamma_s)(x))}
{(\mu_t*\gamma_s)(x)^2}
\prod_{i=1}^2
\gamma_s(x-T_t(X_i))
\d x\,\d\nu_{12}(t).
\end{split}
\end{equation}
The case $\edensity(\rho)=\rho$ is trivial, with both sides of \eqref{eq: energydifferencepoly} equal to zero. Since both sides of \eqref{eq: energydifferencepoly} depend linearly on the energy density and its pressure, the identity therefore holds for every polynomial energy density.

If $P\geq0$ on $[0,C]$, the right-hand side of
\eqref{eq: energydifferencepoly} is nonnegative. Lemma
\ref{lem: polyareenough} therefore gives
\begin{equation}
    \energy{\mu*\gamma_s}
    \leq
    \energy{T_\#\mu*\gamma_s}
\end{equation}
for every internal energy in $\pclass_0$. This proves
\ref{item: positivepressuretheorem}. By the preceding lemma, the same
comparison is equivalent to the existence of the coupling asserted in
\ref{item: positivepressuretransport}.

Under the additional assumptions in \ref{item: positivepressuresmoothtransport}, the construction in \cite[Propositions 2.5 and 2.6 and the proof of Theorem 1.6]{AishwaryaLi25KP} gives a volume-contracting flow transporting $\mu*\gamma_s$ to $(T_t)_\#\mu*\gamma_s$. Its final map has the properties required in \ref{item: positivepressuresmoothtransport}.
\end{proof}
\begin{proof}[Proof of Theorem \ref{thm: GaussianCharacterisationIntro}]
The implication \ref{item: itisgaussian} $\Rightarrow$ \ref{item: ourdispersivecondition} is precisely the divergence computation in \cite[proof of Theorem~1.6]{AishwaryaLi25KP}. Of course, translations of the Gaussian do not affect the conclusion.

Conversely, assume \ref{item: ourdispersivecondition}, and write $g=e^{-\phi}$.
Fix distinct $a,b\in\R^n$ and $w\in\R^n$ such that
\begin{equation}
    \langle w,b-a\rangle<0.
\end{equation}
Consider the two-point measure
\begin{equation}
    \mu=\frac12\left(\delta_0+\delta_{a-b}\right).
\end{equation}
For sufficiently small $\epsilon>0$, the trajectories
\begin{equation}
    T_t(0)=\epsilon t w,
    \qquad
    T_t(a-b)=a-b
\end{equation}
form a smooth continuous contraction. At $t=0$, its velocity field satisfies $v_0(0)=\epsilon w$ and $v_0(a-b)=0$. Evaluating the divergence of the canonical convolved velocity field at $x=a$ gives
\begin{equation}
    \div\widetilde v_0(a)
    =
    -\epsilon
    \frac{g(a)g(b)}{(g(a)+g(b))^2}
    \left\langle
        w,\grad\phi(a)-\grad\phi(b)
    \right\rangle
    \leq0.
\end{equation}
Thus
\begin{equation}
    \left\langle
        w,\grad\phi(a)-\grad\phi(b)
    \right\rangle
    \geq0
\end{equation}
whenever $\langle w,b-a\rangle<0$, and by continuity also when $\langle w,b-a\rangle=0$. Taking both $w$ and $-w$ orthogonal to $b-a$ shows that $\grad\phi(a)-\grad\phi(b)$ is parallel to $b-a$. Taking $w=a-b$ then gives
\begin{equation}
    \grad\phi(a)-\grad\phi(b)
    =
    C_{a,b}(a-b)
\end{equation}
for some $C_{a,b}\geq0$.

Letting $b\to a$ along arbitrary directions shows that every vector is an eigenvector of $\grad^2\phi(a)$. Hence
\begin{equation}
    \grad^2\phi(a)=c(a)I
\end{equation}
for some $c(a)\geq0$. Since $\partial_{ij}\phi=0$ for $i\neq j$, each $\partial_{i}\phi$ depends only on $x_i$. Since $\partial_{11}\phi=\cdots=\partial_{nn}\phi = c(x)$, the function $c(x)$ must be constant. Therefore
\begin{equation}
    \phi(x)
    =
    A+\langle u,x\rangle+\frac{c}{2}\norm{x}^2.
\end{equation}
Integrability of $g=e^{-\phi}$ forces $c>0$, and completing the square shows that $g$ is an isotropic Gaussian density up to translation.
\end{proof}

\section{Gaussian marginalisation, pressure class ascent, and arbitrary contractions} \label{sec: BezdekConnelly}

We will use the following associativity property of marginalisation.

\begin{lemma} \label{lem: marginalisationassociative}
    Let $\sigma_1$ and $\sigma_2$ be probability densities on $\R^{m_1}$ and $\R^{m_2}$, respectively. Whenever the marginalisations below are finite valued,
    \begin{equation}
        \Pi_{\sigma_1\otimes\sigma_2}\energy
        =
        \Pi_{\sigma_2}\Pi_{\sigma_1}\energy.
    \end{equation}
\end{lemma}
\qed

\begin{proof}[Proof of Theorem \ref{thm: BCheatflowintro}]
We write
$\gamma_s^{(2)}(y)=C e^{-\frac{\norm{y}^2}{2s}}$, where
$C=\frac{1}{2\pi s}$. Polar coordinates and the change of variables
$w=C\rho e^{-r^2/(2s)}$ give
\begin{equation}\label{eq: polardecompmarginal}
    (\Pi_{\gamma_s^{(2)}}\edensity)(\rho)
    =
    \frac{1}{C}
    \int_0^{C\rho}\frac{\edensity(w)}{w}\d w.
\end{equation}
The rescaling
\begin{equation}
    \edensity(\rho)
    \longmapsto
    \frac{1}{C}\edensity(C\rho)
\end{equation}
preserves every pressure class, since it sends $P(\rho)$ to
$C^{-1}P(C\rho)$, and also preserves continuity at $0$. By
\eqref{eq: polardecompmarginal}, it therefore suffices to prove the theorem
with $C=1$, which we assume from now on. We will also simply write $\Pi$ for the corresponding $\Pi_{\gamma^{(2)}_{s}}$ in this case.

Fix $j\geq0$. Suppose first that $\edensity\in\pclass_j$, and 
\begin{equation} \label{eq: piufinite}
    \Pi \edensity (\rho) = \int_{0}^{\rho} \frac{\edensity(w)}{w} \d w \in \R.
\end{equation}
Thus $\Pi \edensity$ is continuous at $0$, and since $\frac{\edensity(w)}{w}$ is nondecreasing, $\Pi \edensity$ is also convex. Thus, $\Pi \edensity \in \pclass_{1}$. The assumption that $\Pi \edensity (\rho) \in \R$ also shows that $\edensity (\rho)$ must go to $0$ as $\rho \to 0$, showing the continuity of $\edensity$ at $0$.

Now suppose $j \geq 1$ and $U$ has pressure $P$. A short calculation shows that
\begin{equation}
    \Pi P (\rho) = \int_{0}^{\rho} \frac{P(w)}{w} \d w,
\end{equation}
is the pressure of $\Pi U$. Further, note that there are two equivalent representations of the action of $\Delta_{c}$ on $\Pi P$. First, 
\begin{equation} \label{eq: deltaPicommute}
\begin{split}
    &\Delta_{c} \Pi P (\rho) = \int_{0}^{c \rho} \frac{P(w)}{w} \d w - \int_{0}^{\rho} \frac{P(w)}{w} \d w \\
    &= \int_{0}^{\rho} \frac{P(cw) - P(w)}{w} \d w =          \Pi \Delta_{c} P (\rho),
    \end{split}
\end{equation}
and the second,
\begin{equation} \label{eq: deltagoestointegrallim}
    \Delta_{c} \Pi P (\rho) = \int_{0}^{c \rho} \frac{P(w)}{w} \d w - \int_{0}^{\rho} \frac{P(w)}{w} \d w =
    \int_\rho^{c\rho}\frac{P(w)}{w}\d w
    =
    \int_1^c\frac{P(\rho r)}{r}\d r.
\end{equation}
We use the first representation $j$ times and the second representation once, to show
\begin{equation}
\begin{split}
    \Delta_{c_1}\cdots\Delta_{c_{j+1}}\Pi P(\rho)&=
    \int_1^{c_{1}}
    \Delta_{c_2}\cdots\Delta_{c_{j+1}}
    P(\rho r)\frac{\d r}{r}
    \geq0,
\end{split}
\end{equation}
for $c_{1}, \ldots, c_{j+1} > 1$. Thus, $\Pi \edensity \in \pclass_{j+1}$.

Conversely, let $\edensity\in\pclass_{j+1}$ be continuous at $0$, with pressure $P$. Then, given that $\edensity$ is convex, $P$ is also continuous at $0$. 

Define
\begin{equation}\label{eq: inverseGaussianMarginalisation}
    \Pi^{-1} \edensity{\rho} = \edensity{\rho} + P(\rho) = \rho \edensity'_{+}(\rho),
\end{equation}
which inherits continuity at $0$ from $\edensity$ and $P$. Moreover, 

\begin{equation}
    \Pi \Pi^{-1}\edensity (\rho) = \int_{0}^{\rho} \frac{\edensity(w) + P(w)}{w} \d w = \int_{0}^{\rho} \edensity'_{+}(w) \d w = \edensity(\rho).
\end{equation}
It remains to show that $\Pi^{-1}\edensity \in \pclass_{j}$. This is clear if $j=0$, since $\frac{\Pi^{-1} \edensity{\rho}}{\rho} = \edensity'_{+}(\rho)$ is nondecreasing.

Suppose $j\geq1$. For $c>1$, we start by observing that
$\Delta_c\edensity$ is convex. For $0<a<b$, the pressure identity gives
\begin{equation}
\begin{split}
    \frac{\Delta_c\edensity(b)}{b}
    -
    \frac{\Delta_c\edensity(a)}{a}
    &=
    c\left(
        \frac{\edensity(cb)}{cb}
        -
        \frac{\edensity(ca)}{ca}
    \right)
    -
    \left(
        \frac{\edensity(b)}{b}
        -
        \frac{\edensity(a)}{a}
    \right)\\
    &=
    \int_a^b
    \frac{\Delta_cP(w)}{w^2}\d w.
\end{split}
\end{equation}
Since $\edensity\in\pclass_2$, the function $\Delta_cP$ is
nondecreasing. Consequently, one can show using the last identity that the line at $(a , \Delta_{c}\edensity (a))$ with slope $\frac{\Delta_{c} \edensity(a) + \Delta_{c} P(a)}{a}$ supports the graph of $\Delta_{c} \edensity$. Hence, $\Delta_c\edensity$ is convex. Now
\begin{equation}
    \Pi^{-1}\edensity(\rho)
    =
    \rho\edensity'_+(\rho)
    =
    \lim_{c \to 1^{+}}
    \frac{\Delta_c\edensity(\rho)}{c-1},
\end{equation}
being a pointwise limit of convex functions, is itself convex.

Let $\Pi^{-1}P$ denote the pressure of $\Pi^{-1}\edensity$. Then, for
$r>0$,
\begin{equation}\label{eq: inversepressureGaussianMarginalisation}
\begin{split}
    \Pi^{-1}P(r)
    &=
    \lim_{c\to1^{+}}
    \frac{
        \Pi^{-1}\edensity(cr)
        -
        c\Pi^{-1}\edensity(r)
    }{c-1}\\
    &=
    \lim_{c\to1^{+}}
    \frac{
        \edensity(cr)+P(cr)
        -
        c\edensity(r)-cP(r)
    }{c-1}\\
    &=
    \lim_{c\to1^{+}}
    \frac{
        \edensity(cr)-c\edensity(r)
    }{c-1}
    +
    \lim_{c\to1^{+}}
    \frac{
        P(cr)-cP(r)
    }{c-1}\\
    &=
    P(r)
    +
    \lim_{c\to1^{+}}
    \left(
        \frac{\Delta_cP(r)}{c-1}
        -
        P(r)
    \right)\\
    &=
    \lim_{c\to1^{+}}
    \frac{\Delta_cP(r)}{c-1}.
\end{split}
\end{equation}
Consequently, for $c_1,\ldots,c_j\geq1$,
\begin{equation}
\begin{split}
    \Delta_{c_1}\cdots\Delta_{c_j}\Pi^{-1}P(r)
    &=
    \Delta_{c_1}\cdots\Delta_{c_j}
    \left(
        \lim_{c\to1^{+}}
        \frac{\Delta_cP(r)}{c-1}
    \right)\\
    &=
    \lim_{c\to1^{+}}
    \frac{
        \Delta_{c_1}\cdots\Delta_{c_j}\Delta_cP(r)
    }{c-1}\\
    &=
    \lim_{c\to1^{+}}
    \frac{
        \Delta_c\Delta_{c_1}\cdots\Delta_{c_j}P(r)
    }{c-1}\\
    &\geq0,
\end{split}
\end{equation}
Thus, $\Pi^{-1}\edensity\in\pclass_j$, completing the proof.
\end{proof}

We now combine the above theorem with the classical lifting trick for contractions from \cite{Alexander85}.

\begin{proof}[Proof of Theorem \ref{thm: pressureclassmain}]
First we note that if $\edensity$ is not continuous at $0$, then the corresponding internal energy $\energy{\rho}=-\infty$ for any bounded positive probability density $\rho$ that decays at infinity. Indeed, the $\pclass_1$ condition, namely the convexity of $\edensity$, forces $\lim_{\rho\to0}\edensity(\rho)\leq0$, and failure of continuity at $0$ forces this limit to be strictly negative. Since every $\mu*\gamma_s$ has a bounded, strictly positive density tending to $0$ at infinity, in order to prove \eqref{eq: pressureclassmain}, we can assume that $\edensity$ is continuous at $0$.

As in the proof of Theorem \ref{thm: BCheatflowintro} we can as well, without loss of generality, take the Gaussian density $\gamma_{s}(x) = C e^{- \frac{1}{2s}\norm{x}^{2}}$ such that $C=1$. We continue to use the notation from the proof of
Theorem \ref{thm: BCheatflowintro}

By the nesting of the pressure classes, it suffices to consider
$\energy\in\pclass_k$, where $k=\lceil n/2\rceil$. Applying
Theorem \ref{thm: BCheatflowintro} $k$ times gives
\begin{equation}
    \Pi^{-k}\energy\in\pclass_0,
\end{equation}
where we continue to use the notation $\Pi^{-1}$ from the proof of
Theorem \ref{thm: BCheatflowintro}. Moreover, by Lemma
\ref{lem: marginalisationassociative},
\begin{equation}\label{eq: repeatedGaussianMarginalisation}
    \energy
    =
    \Pi^{k}\Pi^{-k}\energy.
\end{equation}

We use the continuous contraction from $\R^n$ into $\R^{n+2k}$ from \cite[Lemma 1]{BezdekConnelly02}, given by
\begin{equation}
\begin{split}
    \widetilde T_t(x)
    =
    \left(
        \frac{x+T(x)}{2}
        +
        \cos(\pi t)\frac{x-T(x)}{2},
        \sin(\pi t)\frac{x-T(x)}{2},
        0
    \right),
\end{split}
\end{equation}
where the final zero coordinate is present only when $n$ is odd. 

We identify $\R^n$ with the first $n$-dimensional coordinate subspace
of $\R^{n+2k}$, and view $\mu$ as a probability measure on this
subspace. Then Theorem \ref{thm: heatflowpositivepressure}, applied to
$\Pi^{-k}\energy\in\pclass_0$, gives
\begin{equation}
    \Pi^{-k}\energy
    \left(
        \mu*\gamma_s^{(n+2k)}
    \right)
    \leq
    \Pi^{-k}\energy
    \left(
        (\widetilde T_1)_\#\mu*\gamma_s^{(n+2k)}
    \right).
\end{equation}
The density on the left is
\begin{equation}
    (\mu*\gamma_s^{(n)})(x)\gamma_s^{(2k)}(y),
\end{equation}
while the density on the right is
\begin{equation}
    (T_\#\mu*\gamma_s^{(n)})(x)\gamma_s^{(2k)}(y).
\end{equation}
By \eqref{eq: repeatedGaussianMarginalisation}, the preceding
inequality is therefore exactly
\begin{equation}
    \energy{\mu*\gamma_s}
    \leq
    \energy{T_\#\mu*\gamma_s},
\end{equation}
as required.
\end{proof}
\section{Applications to Kneser--Poulsen-type problems} \label{sec: KP}

The more general form of Theorem \ref{thm: heattoKPintro} as promised in Remark \ref{rem: KPforvaryingradiiIntro}, allowing variable radii, is stated below.

\begin{theorem}\label{thm: heattoKPmain}
    Let $K\subseteq\R^n$ be compact, let $T:K\to\R^n$ be a $1$-Lipschitz map, and let $r:K\to(0,\infty)$ be continuous.
    \begin{enumerate}[label=(\roman*),ref=(\roman*)]
        \item\label{item: heattoKPunion}
        Suppose that
        \begin{equation}
            \energy{\mu*\gamma_s}
            \leq
            \energy{T_\#\mu*\gamma_s}
        \end{equation}
        for every $\mu\in\mathcal P(K)$, every $s>0$, and every internal energy $\energy$ with nondecreasing pressure. Then
        \begin{equation}\label{eq: convexdensityHeattoKP}
            \vol{\bigcup_{x\in K}\ball(T(x),r(x))}
            \leq
            \vol{\bigcup_{x\in K}\ball(x,r(x))}.
        \end{equation}

        \item\label{item: heattoKPlayers}
        Suppose that
        \begin{equation}
            \energy{\mu*\gamma_s}
            \leq
            \energy{T_\#\mu*\gamma_s}
        \end{equation}
        for every $\mu\in\mathcal P(K)$, every $s>0$, and every internal energy $\energy$ with nonnegative pressure. Then, for every
        $0<c<\min_{x\in K}r(x)^2$,
        \begin{equation}
        \begin{split}
            &\vol{
                \left(
                    \bigcup_{x\in K}
                    \ball\left(T(x),\sqrt{r(x)^2+c}\right)
                \right)
                \setminus
                \left(
                    \bigcup_{x\in K}
                    \ball\left(T(x),\sqrt{r(x)^2-c}\right)
                \right)
            }\\
            &\leq
            \vol{
                \left(
                    \bigcup_{x\in K}
                    \ball\left(x,\sqrt{r(x)^2+c}\right)
                \right)
                \setminus
                \left(
                    \bigcup_{x\in K}
                    \ball\left(x,\sqrt{r(x)^2-c}\right)
                \right)
            }.
        \end{split}
        \end{equation}
    \end{enumerate}
\end{theorem}

\begin{proof}[Proof of \ref{item: heattoKPunion}]
Choose $\mu\in\mathcal P(K)$ with $\support\mu=K$. For $s>0$, set
\begin{equation}
    N_s
    =
    \int_K e^{\frac{r(x)^2}{2s}}\d\mu(x),
    \qquad
    \d\mu_s(x)
    =
    \frac{e^{\frac{r(x)^2}{2s}}}{N_s}\d\mu(x),
\end{equation}
and write
\begin{equation}
    \gamma_s(y)
    =
    (2\pi s)^{-n/2}e^{-\frac{\norm{y}^2}{2s}}.
\end{equation}
Then
\begin{equation}
\begin{split}
    (\mu_s*\gamma_s)(y)
    &=
    \frac{(2\pi s)^{-n/2}}{N_s}
    \int_K
    e^{\frac{1}{2s}(r(x)^2-\norm{y-x}^2)}
    \d\mu(x),
\end{split}
\end{equation}
and likewise
\begin{equation}
    (T_\#\mu_s*\gamma_s)(y)
    =
    \frac{(2\pi s)^{-n/2}}{N_s}
    \int_K
    e^{\frac{1}{2s}(r(x)^2-\norm{y-T(x)}^2)}
    \d\mu(x).
\end{equation}

Consider the convex energy density
\begin{equation}
    \edensity_s(a)
    =
    -\min\left\{(2\pi s)^{n/2}N_{s} a ,1\right\}.
\end{equation}
The assumed energy inequality gives
\begin{equation}
\begin{split}
&\int_{\R^n}\min\left\{
    \int_K e^{\frac{1}{2s}(r(x)^2-\norm{y-x}^2)}\d\mu(x),1
\right\}\d y\\
&\geq
\int_{\R^n}\min\left\{
    \int_K e^{\frac{1}{2s}(r(x)^2-\norm{y-T(x)}^2)}\d\mu(x),1
\right\}\d y.
\end{split}
\end{equation}

As $s\downarrow0$, the two integrands converge almost everywhere to
\begin{equation}
    \ind_{\cup_{x\in K}\ball(x,r(x))}
    \qquad\text{and}\qquad
    \ind_{\cup_{x\in K}\ball(T(x),r(x))},
\end{equation}
respectively. Indeed, the integral diverges if
$\norm{y-x}<r(x)$ for some $x\in K$ and tends to zero if
$\norm{y-x}>r(x)$ for every $x\in K$. For the remaining points, choose $0<\delta<\min_K r$. These points
are exactly those at distance $\delta$ from the set
$\bigcup_{x\in K}\ball(x,r(x)-\delta)$, hence have Lebesgue measure zero.

To apply dominated convergence, let $R=\max_Kr$. For $s<1$,
\begin{equation}
\begin{split}
&\min\left\{
    \int_K e^{\frac{1}{2s}(r(x)^2-\norm{y-x}^2)}\d\mu(x),1
\right\}\\
&\leq
\ind_{K+R\ball}(y)
+
\bigl(1-\ind_{K+R\ball}(y)\bigr)
e^{\frac12(R^2-d(y,K)^2)},
\end{split}
\end{equation}
and the right-hand side is integrable. The same argument applies with $K$ replaced by $T[K]$. Dominated convergence now gives \eqref{eq: convexdensityHeattoKP}.
\end{proof}

\begin{proof}[Proof of \ref{item: heattoKPlayers}]
Use the same measures $\mu_s$ as above. For $\lambda>0$, set
\begin{equation}
    \Theta_\lambda(a)
    =
    -\min\left\{\frac{a}{\lambda},1\right\}.
\end{equation}
For $\alpha\in\R$, define
\begin{equation}
    \lambda_s(\alpha)
    =
    \frac{(2\pi s)^{-n/2}}{N_s} e^{\frac{\alpha}{2s}}.
\end{equation}
The argument in \ref{item: heattoKPunion}, with $r(x)^2$ replaced by
$r(x)^2-\alpha$, gives
\begin{equation}\label{eq: truncationPowerLimit}
\begin{split}
&\int_{\R^n}
\Theta_{\lambda_s(\alpha)}((\mu_s*\gamma_s)(y))
\d y \longrightarrow
-
\vol{
    \bigcup_{x\in K}
    \ball\left(x,\sqrt{r(x)^2-\alpha}\right)
}
\end{split}
\end{equation}
whenever the radii on the right are positive, and similarly with $x$ replaced by $T(x)$.

Fix $0<c<\min_Kr^2$ and set
\begin{equation}
    \edensity_{s,c}
    =
    \Theta_{\lambda_s(-c)}
    -
    \Theta_{\lambda_s(c)}.
\end{equation}
Since $\lambda_s(-c)<\lambda_s(c)$, a direct check shows that
$\edensity_{s,c}(a)/a$ is nondecreasing. Hence
$\edensity_{s,c}\in\pclass_0$.

The assumed energy inequality therefore gives
\begin{equation}
    \energy_{s,c}(\mu_s*\gamma_s)
    \leq
    \energy_{s,c}(T_\#\mu_s*\gamma_s).
\end{equation}
Using \eqref{eq: truncationPowerLimit} with $\alpha=-c$ and $\alpha=c$ and letting $s\downarrow0$ yields
\begin{equation}
\begin{split}
&-
\vol{
    \left(
        \bigcup_{x\in K}
        \ball\left(x,\sqrt{r(x)^2+c}\right)
    \right)
    \setminus
    \left(
        \bigcup_{x\in K}
        \ball\left(x,\sqrt{r(x)^2-c}\right)
    \right)
}\\
&\leq
-
\vol{
    \left(
        \bigcup_{x\in K}
        \ball\left(T(x),\sqrt{r(x)^2+c}\right)
    \right)
    \setminus
    \left(
        \bigcup_{x\in K}
        \ball\left(T(x),\sqrt{r(x)^2-c}\right)
    \right)
}.
\end{split}
\end{equation}
This is the desired inequality.
\end{proof}

The geometric conclusions stated in Corollary
\ref{cor: classicalKPconsequencesIntro} now follow immediately.
Corollary \ref{cor: classicalKPconsequencesIntro} \ref{item: CsikosKPIntro}, follows from Theorem
\ref{thm: heatflowpositivepressure} and Theorem
\ref{thm: heattoKPmain} \ref{item: heattoKPunion}.
Corollary \ref{cor: classicalKPconsequencesIntro} \ref{item: planarKPIntro}, follows from Theorem
\ref{thm: pressureclassmain} in dimension $2$ and Theorem
\ref{thm: heattoKPmain} \ref{item: heattoKPunion}.
Finally, for obtaining Corollary \ref{cor: classicalKPconsequencesIntro} \ref{item: surfaceKPIntro} from Theorem
\ref{thm: heattoKPmain} \ref{item: heattoKPlayers}, we say a few more words.

\begin{proof}[Proof of Corollary \ref{cor: classicalKPconsequencesIntro} \ref{item: surfaceKPIntro}]
For $z=(z_1,\ldots,z_k)$, set
\begin{equation}
    V_z(u)
    =
    \vol{
        \bigcup_{i=1}^k
        \ball\left(z_i,\sqrt{r_i^2+u}\right)
    }.
\end{equation}
By Theorem \ref{thm: heatflowpositivepressure} and Theorem
\ref{thm: heattoKPmain} \ref{item: heattoKPlayers}, for every
sufficiently small $c>0$,
\begin{equation}
    V_y(c)-V_y(-c)
    \leq
    V_x(c)-V_x(-c).
\end{equation}
Dividing by $2c$ and letting $c\downarrow0$ gives
\begin{equation}
    V_y'(0)\leq V_x'(0).
\end{equation}
For a finite union of balls,
\begin{equation}
    V_z'(0)
    =
    \frac12
    \sum_{i=1}^k
    \frac{\mathcal H^{n-1}(F_i^z)}{r_i};
\end{equation}
see, for example, \cite[Theorem~3]{BezdekConnelly02}. The desired
inequality follows.
\end{proof}

\bibliographystyle{amsplain}
\bibliography{shoulderofgiants}
\end{document}